\documentclass[11pt,leqno,oneside]{amsart}
\usepackage{layout}
\usepackage{mathrsfs,dsfont}
\usepackage{amsmath,amstext,amsthm,amssymb,bbm,color}
\usepackage{comment}
\usepackage{charter}
\usepackage{typearea}
\usepackage{pdfsync}
\usepackage[width=6.5in,height=8.7in]{geometry}

\def\bt{\begin{thm}l.}
\def\et{\end{thm}}
\def\bl{\begin{lem}}
\def\el{\end{lem}}
\def\bd{\begin{defn}}
\def\ed{\end{defn}}
\def\bc{\begin{cor}}
\def\ec{\end{cor}}
\def\bp{\begin{proof}}
\def\ep{\end{proof}}
\def\br{\begin{rem}}
\def\er{\end{rem}}

\newtheorem{thm}{Theorem}[section]
\newtheorem{prop}[thm]{Proposition}
\newtheorem{lem}[thm]{Lemma}
\newtheorem{defn}[thm]{Definition}

\newtheorem{rem}[thm]{Remark}
\newtheorem{cor}[thm]{Corollary}

\numberwithin{equation}{section}

\newcommand{\bthm}{\begin{thm}}
\newcommand{\ethm}{\end{thm}}
\newcommand{\bstp}{\begin{stp}}
\newcommand{\estp}{\end{stp}}
\newcommand{\blemma}{\begin{lemma}}
\newcommand{\elemma}{\end{lemma}}
\newcommand{\bprop}{\begin{prop}}
\newcommand{\eprop}{\end{prop}}
\newcommand{\bpf}{\begin{pf}}
\newcommand{\epf}{\end{pf}}
\newcommand{\bdefn}{\begin{defn}}
\newcommand{\edefn}{\end{defn}}
\newcommand{\brk}{\begin{rmrk}}
\newcommand{\erk}{\end{rmrk}}
\newcommand{\bcrl}{\begin{crl}}
\newcommand{\ecrl}{\end{crl}}

\title[]{Equidistribution for Random Polynomials and Systems of Random Holomorphic Sections}

\address{}

\address{Faculty of Engineering and Natural Sciences, Sabanc{\i} University, \.{I}stanbul, Türkiye}

\email{ozangunyuz@alumni.sabanciuniv.edu}

\date{\today}

\keywords{Random polynomials, equidistribution of zeros, variance, random holomorphic sections}
\subjclass[2020]{Primary: 32A60, 60D05 Secondary: 32U40}

\begin{document}

\author{Ozan Günyüz}

\begin{abstract}
This article addresses an equidistribution problem concerning the zeros of systems of random holomorphic sections of positive line bundles on compact K\"{a}hler manifolds and random polynomials on $\mathbb{C}^{m}$ in the setting of the weighted pluripotential theory. For random polynomials, we consider non-orthonormal bases and prove an equidistribution result which is more general than the ones acquired before for non-discrete probability measures. More precisely, our result demonstrates that the equidistribution holds true even when the random coefficients in the basis representation are not independent and identically distributed (i.i.d.), and moreover, they are not constrained to any particular probability distribution. For random holomorphic sections, by extending the concept of a sequence of asymptotically Bernstein-Markov measures introduced by Bayraktar, Bloom and Levenberg in their recent paper to the setting of holomorphic line bundles over compact Kahler manifolds, we derive a global equidistribution, variance estimate and expected distribution theorems related to the zeros of systems of random holomorphic sections for large tensor powers of a fixed holomorphic line bundle for any codimension k, generalizing a previous result of Bayraktar in his 2016 paper and giving also a positive answer to a question posed in the same paper, asking whether the equidistribution is true for non-homogeneous projective manifolds. For both random holomorphic polynomials on $\mathbb{C}^{m}$ and systems of random holomorphic sections, the variance estimation method detailed in another paper of the author with Bojnik is significant.

  \end{abstract}
\maketitle

%%%%%%%%%%%%%%%%%%%%%%%%%%%%%%%%%%%%%%%%%%%%%%%%%%%%%%%%%%%%%%%%%%%%%%%%%%%%%%%

\begin{quotation}
\center\textit{\small{Dedicated to the memory of Tosun Terzioğlu}}
\end{quotation}

\section{Introduction}\label{s1}

The probabilistic features of the zero sets of random functions with multiple variables captured the interest of mathematicians and physicists alike. This increased focus is especially obvious when dealing with random polynomials with real and complex multivariables. Because of the considerable literature on this subject, we are unable to cite all pertinent references here. As a result, our summary of prior results and the current status of this theory will be brief. For a more comprehensive understanding of both Gaussian and non-Gaussian cases, as well as the progression of polynomial theory, one can turn to works such as \cite{ BL15, BL05, BloomD, ROJ, SHSM, HN08}, among others referenced therein. Prior to the aforementioned advances, it is well-known that Pólya-Bloch, Littlewood-Offord, Kac, Hammersley, and Erdös-Turan lay the groundwork in the study of root distribution of random algebraic equations in a single real variable. For further insights and information, readers can direct their attention to the articles \cite{BlP, Kac43, LO43, HAM56, ET50}. An increasing number of (theoretical) physics papers are dealing with equidistribution and probability-related issues linked to the zeros of complex random polynomials. Foundational studies in this field can be found in the works of \cite{FH, Hann, NV98}.

The broadest framework studied to this point encompasses equidistribution, expected distribution, variance of zero currents of integration of random holomorphic sections, central limit theorem, overcrowding, and hole probability across different probabilistic setups, including Gaussian and non-Gaussian distributions. The foundational and groundbreaking work in this area is attributed to Shiffman and Zelditch (as referenced in \cite{SZ99}). This paper will present equidistribution results using methods developed in the papers \cite{SZ99, Shif, BG}. The main tools employed are the expected distribution and variance estimations of the currents of integration defined via the zero sets of polynomials and holomorphic sections. For similar and other interesting results in diverse probabilistic backgrounds in the general context of holomorphic line bundles over compact K\"{a}hler manifolds, the reader can consult the studies like \cite{BG, BG1, Bay16, BCM, CMM, Shif, SZ08, SZ10}.

The supremum norm $ ||f|| _{D}:=\sup\left\{ \left\vert f\left( z\right) \right\vert :z\in D\right\} $ for a
function $f:D\rightarrow \mathbb{C}$ will be used throughout the paper. We write $d_{n}:=C_{m+n}^{n}=\dim{(\mathcal{P}_{n})}$, where $\mathcal{P}_{n}$ is the vector space of holomorphic polynomials on $\mathbb{C}^{m}$ of degree at most $n$. \thinspace

We say that $F\subset \mathbb{C}^{m}$ is (locally) \textit{pluripolar} if for any $z_{0} \in F$, there is a neighborhood $U$ of $z_{0}$ and a plursiubharmonic function $\psi$ on $U$ such that $F \cap U \subset \{z\in U: \psi(z)=-\infty\}$. Lelong raised the question of whether locally pluripolar sets might be defined by a global plurisubharmonic function, that is, a plurisubharmonic function on $\mathbb{C}^{m}$. These sets are called \textit{globally pluripolar}. By applying techniques from Padé approximation, Josefson solved this problem (\cite{Jos}), so locally and globally pluripolar sets are equivalent. We call sets which are not pluripolar \textit{non-pluripolar}.

The \emph{pluricomplex Green function} of a non-pluripolar compact set $K\subset \mathbb{C}^{m}$ is defined as follows
\begin{equation*}
V_{K}(z):=\sup \{u(z):u|_{K}\leq 0,\ u\in
\mathcal{L}(\mathbb{C}^{m})\},
\end{equation*}%
where \thinspace $\mathcal{L}(\mathbb{C}^{m})$ represents the Lelong class consisting of
all functions\thinspace\ $u$ plurisubharmonic on $\mathbb{C}^{m}$ such that $u(\zeta
)-\ln |\zeta |$ is bounded from above near infinity. The upper semicontinuous regularization of $V_{K}(z)$ is the following $$V^{*}_{K}(z):=\limsup_{\zeta \rightarrow z}{V_{K}(\zeta)}.$$ As is well-known, $V^{*}_{K}(z)\in \mathcal{L}(\mathbb{C}^{m})$
(precisely if $K$ is non-pluripolar, see Corollary 5.2.2 of \cite{Kl}). For more detail about the pluricomplex Green function, we refer the reader to the comprehensive book \cite{Kl}.

A compact set $K$ in $\mathbb{C}^{m}$ is \textit{regular} if $V_{K}\equiv 0$ on $K$ (and therefore $V_{K}$ is continuous on $\mathbb{C}^{m}$). $K$ is said to be \textit{locally regular} if for every $z \in K$, the pluricomplex Green function $V_{K \cap \overline{B(z, r)}}$ is continuous at $z$ for a sufficiently small $r=r(z)>0$. Let $q$ be a weight function, namely, a continuous real-valued function on $K$. Analogous to the unweighted case, the weighted extremal function for the pair $(K, q)$ is defined by $$V_{K, q}(z):=\sup{\{v(z): v\in \mathcal{L}(\mathbb{C}^{m}), v\leq q\,\,\text{on}\,\,K\}}.$$

In a recent paper of Bayraktar, Bloom, and Levenberg (\cite{BBL}), to the best of our knowledge, for the first time, the zero sets of random polynomials in several variables having a representation with respect to a non-orthonormal polynomial basis were investigated. The authors introduced the so-called a sequence of asymptotically and $Z$-asymptotically Chebyshev polynomials building upon the work of Bloom (\cite{Bl01}) on the $\theta$-Chebyshev polynomials, which was developed from the paper of Zakharyuta (\cite{Za}) who defined the directional Chebyshev constants to prove the existence of usual limit for transfinite diameter. The letter $Z$ in front stands in tribute to Zakharyuta and justifies the name $Z$-asymptotically Chebyshev polynomials. They also consider a general basis of polynomials and they provide two probabilistic results in this setting. First, for i.i.d. random coefficients possessing a probability distribution with logarithmically decaying tails and a concentration condition, they prove that the zero currents of integration converges in probability to $dd^c V_{K, q}$, where $K\subset \mathbb{C}^{m}$ is a locally regular compact set and $q$ is a continuous real-valued function on the compact $K$ as the degree $n$ goes to infinity. Second, under the same concentration and i.i.d. assumption on the random coefficients, instead of a probability distribution function, by assuming an expectation hypothesis, it is proven that almost surely the random zero currents of integration converge to $dd^{c} V_{K, q}$  weakly in sense of currents as $p\rightarrow \infty.$ In the unweighted case, these two results are true for a regular compact set $K$ by taking $q\equiv 0$. Recall that when working in the weighted case, for a locally regular compact set $K$, a continuous (real-valued) weight function on it implies the continuity of $V_{K, q}$ on $\mathbb{C}^{m}$, i. e., $V_{K, q}=V^{*}_{K, q}$.

As observed in \cite{Bl01} and \cite{BBL}, for every regular compact set, one can find a sequence of $Z$-asymptotically Chebyshev polynomials. The concept of a sequence of $Z$-asymptotically Chebyshev polynomials is a generalization of many other important polynomial types studied in the literature such as Fekete polynomials associated with an array of Fekete points in a compact set $K \subset \mathbb{C}^{m}$, Leja polynomials associated with a sequence of so-called Leja points in a compact set $K$ and $L^{2}(\mu)$-minimal polynomials for a compact set $K$, where $\mu$ is a Bernstein-Markov measure. For other nice examples, see \cite{BBL}. Inspired by the ideas in \cite{BBL}, for a sequence of $Z$-asymptotically Chebyshev polynomials, the author proved an equidistribution theorem with only a logarithmic moment condition (see Section \ref{s3}) in \cite{Gnyz1} that we will also use in the present paper.

In this study, we establish equidistribution results related to zeros of systems of random \mbox{holomorphic} sections in the setting of high tensor powers of a positive line bundle on a compact Kähler manifold, and zeros of random holomorphic polynomials on $\mathbb{C}^{m}$ under minimal probabilistic assumptions and more general basis constructions by using the weighted pluripotential theory. Specifically, we do not assume any probability distribution for the random coefficients but instead rely on a moment condition (see also \cite{BG}, \cite{BCM}, \cite{Gnyz1}). We extend the concept of asymptotically Bernstein-Markov measures introduced in \cite{BBL} on $\mathbb{C}^{m}$ to the setting of holomorphic line bundles over compact Kähler manifolds. Within this framework, we derive exptected distribution, variance and equidistribution for any codimension, answering positively a question posed in a previous research by \cite{Bay16}, showing that equidistribution results also hold on non-homogeneous projective manifolds. This expands the applicability of such theorems. We follow a unified approach for random systems of holomorphic sections. The present work also examines specific cases, such as the Gaussian and the Fubini-Study measures, verifying that the moment condition holds universally under these distributions. The analysis also includes random polynomials with heavy-tailed and logarithmically decaying distributions investigated in previous papers such as \cite{BL15} and \cite{Bay16}.

\vspace{5mm}

This study is dedicated to the memory of Tosun Terzioğlu, who made significant and \mbox{distinguished} contributions to functional analysis, with profound respect and remembrance.

\section{Some Auxiliary Deterministic Results}\label{s2}

Following \cite{BBL}, we will work in the weighted setting. Let $K \subset \mathbb{C}^{m}$ be a locally regular compact set and $q$ a continuous real-valued weight function on $K$. We take a general basis $\{p_{nj}\}^{d_{n}}_{j=1}$ of $\mathcal{P}_{n}$ with $\|p_{nj}e^{-nq} \|_{K} \leq 1,\,\,j=1, 2, \ldots, d_{n}$ and $n=1, 2, \ldots$. For our purposes, one standing assumption will be that $p_{n1}(z) \equiv 1$ for any $n\in \mathbb{N}$. We consider the Bergman-type functions \begin{equation}\label{bergty} \Gamma_{n}(z):= \sum_{j=1}^{d_{n}}{|p_{nj}(z)|^{2}}.\end{equation}We make a simple observation on $\Gamma_{n}$. To begin with, recall the well-known Siciak-Zakharyuta formulation of the weighted extremal function $V_{K, q}$ (see, for example, the appendix by T. Bloom in \cite{SaffTotik}) \begin{equation}\label{siza}
                    V_{K, q}(z)=\sup{\{\frac{1}{\deg{(p)}}\log{|p(z)|}: p\in \cup_{n}{\mathcal{P}_{n}}, \|pe^{-\deg{(p)}\,q}\|_{K}\leq 1\}}.
                  \end{equation}Now for any polynomial $p_{nj}$ of degree $n$ satisfying $\|p_{nj}e^{-nq} \|_{K} \leq 1$, from (\ref{siza}), we have first \begin{equation*}
                                                                           |p_{nj}(z)| \leq e^{n V_{K,q}(z)}.
                                                                         \end{equation*}This gives
\begin{equation*}
\Gamma_{n}(z)=\sum_{j=1}^{d_{n}}{|p_{nj}(z)|^{2}} \leq d_{n} e^{2n V_{K,q}(z)},
\end{equation*}which concludes, when passing to the logarithm, the following

\begin{equation*}\label{siza2}
  \frac{1}{n}\log{\Gamma_{n}(z)} \leq \frac{1}{n}\log{d_{n}} + 2 V_{K, q}(z).
\end{equation*}Taking the upper limit as $n \rightarrow \infty$, we get \begin{equation}\label{siza3}
                                                                         \limsup_{n\rightarrow \infty}{\frac{1}{2n}\log{\Gamma_{n}(z)}} \leq V_{K, q}(z) \,\,\text{for}\,\,\text{all}\,\,z\in \mathbb{C}^{m}.
                                                                       \end{equation}
It follows from (\ref{siza3}) and the continuity of the extremal function $V_{K, q}$ that \begin{equation}\label{siza4}
                                     [\limsup_{n\rightarrow \infty}{\frac{1}{2n}\log{\Gamma_{n}(z)}}]^{*} \leq V_{K, q}(z).
                                   \end{equation}Since $V_{K, q}$ is continuous, the estimation (\ref{siza4}) yields that $\{\frac{1}{2n}\log{\Gamma_{n}}\}$ is locally uniformly bounded.

The ensuing lemma concerning subharmonic functions on $\mathbb{R}^{m}$ from \cite{BBL} will be used to prove the $L^{1}_{loc}(\mathbb{C}^{m})$-convergence for Bergman-type functions given above.

\begin{lem}\label{hörm}
Given a sequence $\{v_{n}\}_{n=1}^{\infty}$ of subharmonic functions on a domain $G\subset \mathbb{R}^{m},\,\,m \geq 2$, which are locally uniformly bounded above on $G$, and a continuous and subharmonic function $v$ on $G$, assume that the following two conditions hold \begin{itemize}
 \item[(i)] $[\limsup_{n\rightarrow \infty}{v_{n}(z)}]^{*} \leq v(z)$ for all $z\in G$.
  \item[(ii)] There is a countable dense set $\{z_{s}\}$ in $G$ such that \begin{equation*}
   \lim_{n\rightarrow \infty}{v_{n}(z_{s})}=v(z_{s}),\,\,s=1, 2, \ldots.
   \end{equation*}
   \end{itemize} Then $v_{n} \rightarrow v$ in $L^{1}_{loc}(G)$.
\end{lem}

In the unweighted case, it is proven in [\cite{BBL}, Proposition 2.3]  that, by making use of a Zakharyuta-Siciak type theorem of Bloom (\cite{Bl01}, Theorem 4.2) for $Z$-asymptotically Chebyshev sequences for compact sets in $\mathbb{C}^{m}$ and a diagonalization argument, when we are given a subsequence $L$ of $\mathbb{N}$, it is possible to find another subsequence $L' \subset L$ and a countable dense subset of points $\{z_{s}\}\,(s=1, 2, \ldots) $\,in\, $\mathbb{C}^{m}$ such that the following holds $$\lim_{n\rightarrow \infty,\,n\in L'}{\frac{1}{2n}\log{\Gamma_{n}(z_{s})}}=V_{K}(z_{s}), \,s=1, 2, \ldots$$Thus, combining this last information with Lemma \ref{hörm}, the authors prove the $L^{1}_{loc}(\mathbb{C}^{m})$-convergence of $\{\frac{1}{2n}\log{\Gamma_{n}}\}$ to $V_{K}$ for sequences of $Z$-asymptotically Chebyshev polynomials (Corollary 2.6 there). For a general polynomial basis, we lack such an information as above, however if we assume the existence of a countable dense subset just like the authors did in \cite{BBL}, we have the following proposition for the sequence $\{\frac{1}{2n}\log{\Gamma_{n}}\}$.

\begin{prop}\label{Locc1}
Let $\Gamma_{n}(z)$ be as in (\ref{bergty}). Assume that there is a countable dense set $\{z_{s}\},\,\,s=1, 2, \ldots,$ in $\mathbb{C}^{m}$ such that
\begin{equation*}
 \lim_{n\rightarrow \infty}{\frac{1}{2n}\log{\Gamma_{n}(z_{s})}}=V_{K, q}(z_{s}), \,s=1, 2, \ldots.
\end{equation*} Then $\frac{1}{2n}\log{\Gamma_{n}} \rightarrow V_{K, q}$ in $L^{1}_{loc}(\mathbb{C}^{m})$.
\end{prop}

\begin{proof}
This follows from the relation (\ref{siza4}), the assumption and Lemma \ref{hörm}.
\end{proof}

\section{Probabilistic Model}\label{s3}

The framework for our forthcoming discussions is presented as follows: We base our approach on papers such as \cite{BCM}  (also refer to \cite{BL15}) to explain our method of randomizing the space $\mathcal{P}_{n}$. Suppose that $K\subset \mathbb{C}^{m}$ is a locally regular compact set and $q$ is a continuous real-valued weight function on $K$. Let us take $\{p_{nj}\}_{j=1}^{d_{n}}$ as a general basis for $\mathcal{P}_{n}$ with a condition $||p_{nj} e^{-nq}||_{K} \leq 1$ (For the unweighted case, we take the compact set $K$ to be regular and $||p_{nj}||_{K} \leq 1$ ). Consequently, for each polynomial $f_{n} \in \mathcal{P}_{n}$ of degree $n$, it can be written as \begin{equation}\label{repch}f_{n}(z)=\sum_{j=1}^{d_{n}}{a_{j}^{(n)} p_{nj}(z)}:=\langle a^{(n)}, p^{(n)}(z) \rangle \in \mathcal{P}_{n},\end{equation} where $a^{(n)}:=(a^{(n)}_{1}, \ldots, a^{(n)}_{d_{n}})\in \mathbb{C}^{d_{n}}$ and $p^{(n)}(z):=(p_{n1}(z), \ldots, p_{nd_{n}}(z)) \in \mathcal{P}_{n}^{d_{n}}$. We identify the space $\mathcal{P}_{n}$ with $\mathbb{C}^{d_{n}}$ and equip it with a probability measure $\mu_{n}$ that puts no mass on pluripolar sets and meets the moment condition stated below:

There exist a constant $\alpha\geq 2$ and for every $n\geq 1$ constants $D_{n}=o(n^{\alpha})>0$ such that
\begin{equation} \label{moment} \int_{\mathbb{C}^{d_{n}}}{\bigl\lvert\log{|\langle a, v \rangle|\bigr\rvert^{\alpha}}d\mu_{n}(a)}\leq D_{n}\end{equation}for every $v\in \mathbb{C}^{d_{n}}$ \, with\, $\| v \|=1$. Hence $(\mathcal{P}_{n}, \mu_{n})$ is the probability space comprising the random polynomials. We also consider the infinite product probability measure $\mu_{\infty}$ induced by $\mu_{n}$, that is $\mu_{\infty}= \prod_{n=1}^{\infty}{\mu_{n}}$ on the product space $\prod_{n=1}^{\infty}{\mathcal{P}_{n}}$: \begin{equation*}
                                                      (\mathcal{P}_{\infty}, \mathbf{\mu}_{\infty})=(\prod_{n=1}^{\infty}{\mathcal{P}_{n}}, \prod_{n=1}^{\infty}{\mu_{n}}).
                                                    \end{equation*} These probability spaces varying with the degree $n$ depend on the choice of basis, however the equidistribution of zeros of polynomials will be independent of the basis chosen, as Theorem \ref{equidc1} corroborates.

Throughout this paper, we shall be working with real-valued test forms of bidegree $(p, q)$ on $\mathbb{C}^{m}$ and the space of these forms will be denoted by $\mathcal{D}^{p, q}_{\mathbb{R}}(\mathbb{C}^{m})$. Let $f\in \mathcal{P}_{n}$, we use the symbol $Z_{f}$ for the zero set of $f$, i.e., $Z_{f}:=\{z\in \mathbb{C}^{m}: f(z)=0\}$. We then consider the random current of integration over $Z_{f}$, in symbols $[Z_{f}]$, defined as follows: Given a test form \,$\varphi\in \mathcal{D}^{m-1, m-1}_{\mathbb{R}}(\mathbb{C}^{m})$  $$ \langle [Z_{f}], \varphi \rangle := \int_{\mathrm{Reg}Z_{f}}{\varphi},$$where $\mathrm{Reg}Z_{f}$ is the set of regular points of $Z_{f}$. By a result of Lelong (see, e.g., \cite{Dem12}, Chapter 3, Theorem 2.7), $[Z_{f}]$ is a closed positive $(1, 1)$-current on $\mathbb{C}^{m}$.
The expectation and the  variance of the random current of integration $[Z_{f}]$ are given by \begin{align}\label{expvari}
  \mathbb{E} \langle [Z_{f}], \varphi \rangle&:=\int_{\mathcal{P}_{n}}{\langle [Z_{f}], \varphi\rangle\, d\mu_{n}(f)} \nonumber \\  \mathrm{Var} \langle [Z_{f}], \, \varphi \rangle &:= \mathbb{E}\langle [Z_{f}], \, \varphi \rangle^{2}- (\mathbb{E}\langle [Z_{f}], \, \varphi \rangle)^{2},\end{align} where $\varphi\in \mathcal{D}^{m-1, m-1}_{\mathbb{R}}(\mathbb{C}^{m})$ and $\mu_{n}$ is the probability measure on $\mathcal{P}_{n}$ that arises from the identification of $\mathcal{P}_{n}$ with $\mathbb{C}^{d_{n}}$. As will be shown in Section 4, variance and expectation are well-defined quantities and expectation can be viewed as a positive closed current.

Notice that the moment condition (\ref{moment}) differs somewhat from the one presented in \cite{BCM} (p.3, assumption ($B$)). This modification is made to assure for the variance of a random current of integration along the zero set of a polynomial to be well-defined (see section 4 for details).

We often utilize the well-known Poincar\'{e}-Lelong formula \begin{equation}\label{pole}
                                                            [Z_{f}]= dd^c \log{|f|}.
                                                          \end{equation}
Throughout the paper, we use the normalized form $dd^{c}=\frac{i}{\pi} \partial \overline{\partial}$ and also take into consideration the random currents of integration by normalizing them with the degree of the polynomial, that is, given that $f_{n}\in \mathcal{P}_{n}$ of degree $n$ as in (\ref{repch}), $[\widehat{Z_{f_{n}}}]:=\frac{1}{n}[Z_{f_{n}}]$.

We would like to point out that for the sake of simplicity, we only use test forms on $\mathbb{C}^{m}$ in this paper. However, our findings are also valid for continuous forms on $\mathbb{C}^{m}$ with compact support because of the density of test forms in compactly supported continuous forms.

\section{Equidistribution in $1$-codimensional Case}\label{s4}

In this section, we will begin by examining the expected distribution of zeros of a polynomial of degree $n$ as expressed in (\ref{repch}). Then, we will deal with the variance estimate of the random current of integration $[\widehat{Z_{f_{n}}}]$, and finally we will prove the equidistribution theorem for these random currents. The underlying ideas that we employed in the proofs draw inspiration from the works of \cite{SZ99} and \cite{BG}. After each of these three results in the weighted setting, we also provide their unweighted counterparts with a regular compact set as corollaries whose proofs are verbatim to that of the weighted versions.
\subsection{Expected Distribution of Zeros}

\begin{lem}\label{expw}
Let $K\subset \mathbb{C}^{m}$ be a locally regular compact subset, a continuous real-valued weight function $q$ on $K$ be given. Suppose that there is a countable dense set $\{z_{s}\}$ in $\mathbb{C}^{m}$ such that  \begin{equation}\label{count}
                                                                                \lim_{n\rightarrow \infty}{\frac{1}{2n}\log{\Gamma_{n}(z_{s})}}=V_{K, q}(z_{s}) \,\,\,\,s=1, 2, \ldots.
                                                                              \end{equation}  Then for a random polynomial $f_{n}(z)=\sum_{j=1}^{d_{n}}{a_{j}^{(n)} p_{nj}(z)}$, we have
\begin{equation}\label{expd}
 \mathbb{E}[\widehat{Z_{f_{n}}}] \longrightarrow dd^c V_{K, q}
\end{equation} in the weak* topology of currents as $n \rightarrow \infty$.
\end{lem}

\begin{proof}
We start with writing the following unit vectors in $\mathbb{C}^{d_{n}}$ by (\ref{repch}), \begin{equation}\label{makeu}
\beta^{(n)}(z):=\frac{1}{\sqrt{\Gamma_{n}(z)}}p^{(n)}(z)=\big(\frac{p_{n1}(z)}{\sqrt{\Gamma_{n}(z)}}, \ldots, \frac{p_{nd_{n}}(z)}{\sqrt{\Gamma_{n}(z)}}\big).
\end{equation}Now we observe that \begin{equation}\label{decomp}
\frac{1}{n}\log{|f_{n}(z)|}=\frac{1}{n}\log{|\langle a^{(n)}, \beta^{(n)}(z) \rangle|} + \frac{1}{2n}\log{\Gamma_{n}(z)},
\end{equation} where, as before, $a^{(n)}=(a^{(n)}_{1}, \ldots, a^{(n)}_{d_{n}})\in \mathbb{C}^{d_{n}}$ and $p^{(n)}(z):=(p_{n1}(z), \ldots, p_{nd_{n}}(z))$.
Pick $\varphi \in \mathcal{D}^{m-1, m-1}_{\mathbb{R}}(\mathbb{C}^{m})$. Due to the definition of expectation, the Poincar\'{e}-Lelong formula (\ref{pole}), the identification of $\mathcal{P}_{n}$ and the Fubini-Tonelli's theorem, \begin{equation}\label{expd2}
  \mathbb{E}\langle [\widehat{Z_{f_{n}}}], \varphi \rangle= \int_{\mathbb{C}^{d_{n}}}{\langle\frac{1}{2n}dd^{c} \log{\Gamma_{n}}, \varphi \rangle d\mu_{n}(a^{(n)})} + \frac{1}{n} \int_{\mathbb{C}^{m}}{\int_{\mathcal{P}_{n}}{\log{|\langle a^{(n)}, \beta^{(n)}(z)\rangle|}}d\mu_{n}(f_{n}) dd^{c}\varphi(z)}. \end{equation}By the moment condition (\ref{moment}) and the H\"{o}lder's inequality, the second term in (\ref{expd2}) does have a bound from above,
  \begin{equation}\label{secondterm}
 \frac{1}{n} \int_{\mathbb{C}^{m}}{\int_{\mathbb{C}^{d_{n}}}{\big| \log{|\langle a^{(n)}, \beta^{(n)}(z)\rangle|} \big|}d\mu_{n}(a^{(n)}) dd^{c}\varphi(z)} \leq \frac{D_{n}^{1/\alpha}}{n} C_{\varphi}, \end{equation} where $C_{\varphi}$ is a finite constant depending on the form $\varphi$, which has a compact support in $\mathbb{C}^{m}$. In more concrete terms, $C_{\varphi}$ can be thought of as the sum of the supremum norms of the coefficients of the form $dd^{c}\varphi$. By the relation (\ref{siza4}) and the assumption of countable dense subset, Lemma \ref{hörm} gives that $\frac{1}{2n}\log{\Gamma_{n}} \rightarrow V_{K, q}$ in $L^{1}_{loc}(\mathbb{C}^{m})$. All in all, when $n \rightarrow \infty$, the second term goes to zero owing to the inequality (\ref{secondterm}), and the first term converges to $dd^{c}V_{K, q}$ in the weak* topology by Lemma \ref{hörm}, which ends the proof.\end{proof}

\begin{cor}
Let $K\subset \mathbb{C}^{m}$ be a regular compact subset. Suppose that there is a countable dense set $\{z_{s}\}$ in $\mathbb{C}^{m}$ such that  \begin{equation}\label{count}
                                                                                \lim_{n\rightarrow \infty}{\frac{1}{2n}\log{\Gamma_{n}(z_{s})}}=V_{K}(z_{s}) \,\,\,\,s=1, 2, \ldots.
                                                                              \end{equation}  Then for a random polynomial $f_{n}(z)=\sum_{j=1}^{d_{n}}{a_{j}^{(n)} p_{nj}(z)}$ with $||p_{nj}||_{K} \leq 1$, we have
\begin{equation}\label{expd}
 \mathbb{E}[\widehat{Z_{f_{n}}}] \longrightarrow dd^c V_{K}
\end{equation} in the weak* topology of currents as the degree $n \rightarrow \infty$.
\end{cor}

\subsection{Variance Estimate}
We will bound the variance of a normalized random current of integration $[\widehat{Z_{f_{n}}}]$ associated with the zero set of a polynomial $f_{n}\in \mathcal{P}_{n}$ from above as in \cite[Theorem 3.1]{BG1}. Before finding bounds for two different terms of variance given in (\ref{expvari}), we note that $\mathbb{E}\langle [Z_{f_{n}}], \varphi \rangle$ is bounded. Indeed, since the expectation is real-valued, and so, if we take the absolute value of both sides of the expression (\ref{expd2}) and by using the local uniform boundedness of $\{\frac{1}{2n}\log{\Gamma_{n}}\}$ (discussion following (\ref{siza4}) in Section \ref{s2}) for the first integral and (\ref{secondterm}) for the second integral on the right-hand side of (\ref{expd2}), then we see that $|\mathbb{E}\langle [\widehat{Z_{f_{n}}}], \varphi \rangle|$ and thus $\mathbb{E}\langle [\widehat{Z_{f_{n}}}], \varphi \rangle$ is bounded. With these findings and (\ref{expd2}), we have
\begin{equation}\label{excurr}
	\mathbb{E}[\widehat{Z_{f_{n}}}]=\frac{1}{2n}dd^{c}\log \Gamma_{n}(x)+\frac{1}{n}dd^{c} \big(\int_{f_{n} \in \mathcal{P}_{n}}{{\log{|\langle a^{(n)}, \beta^{(n)}(z)\rangle|}}d\mu_{n}} \big),
\end{equation}which shows that $\mathbb{E}\langle [\widehat{Z_{f_{n}}}], \phi \rangle= \langle \mathbb{E}[\widehat{Z_{f_{n}}}], \phi \rangle$, in other words, the expected value can be regarded as a well-defined current. It is also referred to as the expected current of integration, which is a positive, closed $(1, 1)$-current.

Also observe that in Lemma \ref{expw}, the exponent $\alpha$ does not necessarily need to be greater than or equal to $2$. The condition $\alpha \geq 1$ is adequate for this case, however, for the variance estimation, we must have the condition $\alpha \geq 2$.

\begin{thm}\label{main}

Suppose $K\subset \mathbb{C}^{m}$ is a locally regular compact set, $q$ is a continuous real-valued weight function defined on $K$. Let $f_{n} \in \mathcal{P}_{n}$ be as in (\ref{repch}). Then, for a form $\varphi \in \mathcal{D}^{m-1, m-1}_{\mathbb{R}}(\mathbb{C}^{m})$, we have the following variance estimation of the random current of integration $[\widehat{Z_{f_{n}}}]$,
 \begin{equation}\label{varas}
    \mathrm{Var}{\langle [\widehat{Z_{f_{n}}}], \varphi \rangle}\leq C_{\varphi}^{2}\,D_{n}^{\frac{2}{\alpha}} \frac{1}{n^{2}},
\end{equation}where $C_{\varphi}$ is a constant depending on the test form $\varphi$.
\end{thm}
\begin{proof}
   Let $\varphi \in \mathcal{D}^{m-1, m-1}_{\mathbb{R}}(\mathbb{C}^{m})$. First of all, by (\ref{repch}) for $f_{n}$ and the Poincar\'{e}-Lelong formula (\ref{pole}), we get
\begin{equation}\label{cor}
   \mathbb{E}\langle [\widehat{Z_{f_{n}}}], \, \varphi \rangle^{2}= \frac{1}{n^{2}} \int_{\mathcal{P}_{n}}\int_{\mathbb{C}^{m}}\int_{\mathbb{C}^{m}}{\log{|\langle a^{(n)}, p^{(n)}(z) \rangle|}\log{|\langle a^{(n)}, p^{(n)}(w) \rangle|}dd^{c}\varphi(z)dd^{c}\varphi(w)d\mu_{n}(f_{n})},\end{equation} where $p^{(n)}(z):=(p_{n1}(z), \ldots, p_{nd_{n}}(z))$ and  $a^{(n)}=(a^{(n)}_{1}, \ldots, a^{(n)}_{d_{n}})\in \mathbb{C}^{d_{n}}$ as defined before. By the relation (\ref{decomp}), the integrand of (\ref{cor}) takes the following form
   \begin{align}
\frac{1}{4n^{2}}\log{\Gamma_{n}(z)}\log{\Gamma_{n}(w)}+\frac{1}{2n^{2}} \log{\Gamma_{n}(z)}\log{| \langle a^{(n)}, \beta^{(n)}(w)\rangle|} + \frac{1}{2n^{2}} \log{\Gamma_{n}(w)}\log{| \langle a^{(n)}, \beta^{(n)}(z)\rangle|}
\label{cordec}  \\  +\frac{1}{n^{2}}\log{| \langle a^{(n)}, \beta^{(n)}(z)\rangle|}\log{| \langle a^{(n)}, \beta^{(n)}(w)\rangle|}\label{cordec1}.\end{align}

For the second term of variance, by expanding the expectation expression (\ref{expd2}), we have  \begin{equation*}
                                                                                               \big(\mathbb{E}\big\langle [\widehat{Z_{f_{n}}}], \varphi \big\rangle\big)^{2}=J_{1}+ 2 J_{2}+J_{3},
                                                                                              \end{equation*}where \begin{equation} \label{J1}J_{1}= \big(\frac{1}{4n^{2}}\int_{\mathcal{P}_{n}}\int_{\mathbb{C}^{m}}{\log{\Gamma_{p}(z)}\,dd^{c}\varphi(z)d\mu_{n}(f_{n})} \big)^{2}\end{equation} \begin{equation}\label{J2}
                                                                                                                            J_{2}= \big(\frac{1}{2n}\int_{\mathcal{P}_{n}}\int_{\mathbb{C}^{m}}{\log{\Gamma_{p}(z)}\,dd^{c}\varphi(z)d\mu_{n}(f_{n})} \big) \big( \frac{1}{n} \int_{\mathcal{P}_{n}}\int_{\mathbb{C}^{m}}{\log{|\langle a^{(n)}, \beta^{(n)}(z)\rangle|}}dd^{c}\varphi(z)d\mu_{n}(f_{n})\big)
                                                                                                                           \end{equation} and \begin{equation}\label{J3}
                                                                                                                             J_{3}= \big(\frac{1}{n^{2}} \int_{\mathcal{P}_{n}}\int_{\mathbb{C}^{m}}{\log{|\langle a^{(n)}, \beta^{(n)}(z)\rangle|}}dd^{c}\varphi(z)d\mu_{n}(f_{n})\big)^{2}
                                                                                                                           \end{equation}
Note that all of the integrals $J_{1}, J_{2}$ and $J_{3}$ are finite since $\mathbb{E}\langle \widehat{[Z_{f_{n}}}], \varphi \rangle$ is bounded.

Considering now the four integrands given in (\ref{cordec}) and (\ref{cordec1}), we first have $\mathbb{E}\langle \widehat{[Z_{f_{n}}}], \varphi \rangle^{2}= I_{1} + 2 I_{2} + I_{3}$, where
 \begin{equation}\label{onevar} I_{1}= \int_{\mathcal{P}_{n}}\int_{\mathbb{C}^{m}}\int_{\mathbb{C}^{m}}{\frac{1}{2n}\log{\Gamma_{n}(z)} \frac{1}{2n}\log{\Gamma_{n}(w)}dd^{c}\varphi(z)dd^{c}\varphi(w) d\mu_{n}(f_{n})},\end{equation}

 \begin{equation}\label{twovar}
   I_{2}= \int_{\mathcal{P}_{n}}\int_{\mathbb{C}^{m}}\int_{\mathbb{C}^{m}}{\frac{1}{2n}\log{\Gamma_{n}(z)} \frac{1}{n}\log{|\langle a^{(n)}, \beta^{(n)}(w)\rangle|}dd^{c}\varphi(z)dd^{c}\varphi(w) d\mu_{n}(f_{n})},
 \end{equation}
 (The second  term $\frac{1}{2n^{2}} \log{\Gamma_{n}(z)}\log{| \langle a^{(n)}, \beta^{(n)}(w)\rangle|}$  and the third one   $\frac{1}{2n^{2}} \log{\Gamma_{n}(w)}\log{| \langle a^{(n)}, \beta^{(n)}(z)\rangle|}$ in (\ref{cordec}) are actually the integrands that yield the same result) and \begin{equation}\label{threevar}
                     I_{3}=\int_{\mathcal{P}_{n}}\int_{\mathbb{C}^{m}}\int_{\mathbb{C}^{m}}{\frac{1}{n}\log{|\langle a^{(n)}, \beta^{(n)}(z)\rangle|} \frac{1}{n} \log{|\langle a^{(n)}, \beta^{(n)}(w)\rangle|}}dd^{c}\varphi(z) dd^{c} \varphi(w) d\mu_{n}(f_{n}). \end{equation}
 From  the locally uniform boundedness of $\{\frac{1}{2n} \log{\Gamma_{n}}\}$ (see the arguments in Section \ref{s2}), the moment assumption (\ref{moment}) and the Fubini-Tonelli's theorem, we see that $I_{1}, I_{2}$ and $I_{3}$ are all finite, what is more, we get $I_{1}=J_{1}$ and $I_{2}=J_{2}$. Hence, the only integrals left are $I_{3}$ and $J_{3}$, which are not always equal to each other. Thus, we have \begin{equation}\label{varip}
    \mathrm{Var}\langle [\widehat{Z_{f_{n}}}], \varphi \rangle = I_{3}-J_{3},
    \end{equation}so it will suffice to estimate the term $I_{3}$ from above to get the variance estimation. To do this, we apply Hölder's inequality twice using the fitting exponents. Then, by the Hölder's inequality with $\frac{1}{\alpha}+ \frac{1}{\theta}=1$ and proceeding exactly in the same way as above, where $\alpha \geq 2$ is the exponent in the moment condition (\ref{moment}), one first has

\begin{equation}\label{varterm3}I_{3} \leq\int_{\mathbb{C}^{m}}\int_{\mathbb{C}^{m}}{dd^{c}\varphi(z) dd^{c}\varphi(w)} \frac{1}{n^{2}} \int_{\mathbb{C}^{d_{n}}}{|\log{|\langle a^{(n)}, \beta^{(n)}(z) \rangle|}||\log{|\langle a^{(n)}, \beta^{(n)}(w) \rangle|}| d\mu_{n}(a^{(n)})}.\end{equation} The right-hand side of this last inequality is less than or equal to the following by the Hölder's inequality $$\int_{\mathbb{C}^{m}}\int_{\mathbb{C}^{m}}{dd^{c}\varphi(z) dd^{c}\varphi(w)}\frac{1}{n^{2}} \Big\{ \int_{\mathbb{C}^{d_{n}}}{|\log{|\langle a^{(n)}, \beta^{(n)}(z) \rangle|}|^{\alpha}d\mu_{n}(a^{(n)})}\Big\}^{\frac{1}{\alpha}} \Big\{ \int_{\mathbb{C}^{d_{n}}}{|\log{|\langle a^{(n)}, \beta^{(n)}(z) \rangle|}|^{\theta}d\mu_{n}(a^{(n)})}\Big\}^{\frac{1}{\theta}},$$ which gives that

$$I_{3} \leq \int_{\mathbb{C}^{m}}\int_{\mathbb{C}^{m}}{dd^{c}\varphi(z) dd^{c}\varphi(w)} \frac{1}{n^{2}} D_{n}^{\frac{1}{\alpha}} \Big\{ \int_{\mathbb{C}^{d_{n}}}{|\log{|\langle a^{(n)}, \beta^{(n)}(z) \rangle|}|^{\theta}d\mu_{n}(a^{(n)})}\Big\}^{\frac{1}{\theta}}.$$

    We have to apply Hölder's inequality to the innermost integral once more as we mentioned. Here, the stipulation that $\alpha \geq 2$ (therefore, $\alpha \geq 2 \geq \theta$) is pivotal, since it permits us to reuse the Hölder's inequality, resulting in, \begin{equation}\label{höld2}
                       I_{3} \leq \int_{\mathbb{C}^{m}}\int_{\mathbb{C}^{m}}{dd^{c}\varphi(z) dd^{c}\varphi(w)} \frac{1}{n^{2}} D_{n}^{\frac{2}{\alpha}} \leq \frac{1}{n^{2}}C_{\varphi}^{2} D_{n}^{\frac{2}{\alpha}},
                     \end{equation} which concludes \begin{equation}\label{vars}
                      \mathrm{Var}{\langle [\widehat{Z_{f_{n}}}], \varphi \rangle}\leq C_{\varphi}^{2}\,\frac{D_{n}^{2/\alpha}}{n^{2}},
                    \end{equation}thereby finalizing the variance estimate of the random current of integration $[\widehat{Z_{f_{n}}}]$.\end{proof}
\begin{cor}
Let $K\subset \mathbb{C}^{m}$ be a regular compact subset, and let a continuous real-valued weight function $q$ on $K$ be given. Let $f_{n} \in \mathcal{P}_{n}$ be as in (\ref{repch}) with $||p_{nj}||_{K} \leq 1$. For  $\varphi \in \mathcal{D}^{m-1, m-1}_{\mathbb{R}}(\mathbb{C}^{m})$, we have the following variance estimation of the random current of integration $[\widehat{Z_{f_{n}}}]$,
 \begin{equation}\label{varas}
    \mathrm{Var}{\langle [\widehat{Z_{f_{n}}}], \varphi \rangle}\leq C_{\varphi}^{2}\,D_{n}^{\frac{2}{\alpha}} \frac{1}{n^{2}},
\end{equation}where $C_{\varphi}$ is a constant depending on the test form $\varphi$.
\end{cor}

We are now ready to prove the equidistribution theorem.

\begin{thm}\label{equidc1}
 Let $K\subset \mathbb{C}^{m}$ and $q$ be as before. Let there be a countable dense set $\{z_{s}\}$ in $\mathbb{C}^{m}$ such that \begin{equation}\label{count}
                                                                                \lim_{n\rightarrow \infty}{\frac{1}{2n}\log{\Gamma_{n}(z_{s})}}=V_{K, q}(z_{s}) \,\,\,\,s=1, 2, \ldots.
                                                                              \end{equation} Then for $\mu_{\infty}$-almost every sequence $\textbf{f}=\{f_{n}\}_{n=1}^{\infty}$, \begin{equation}\label{equidc}
                                                      [\widehat{Z_{f_{n}}}] \longrightarrow dd^{c} V_{K, q}
                                                    \end{equation}in the weak* topology of currents as $n\rightarrow \infty$ in case $\sum_{n=1}^{\infty}{\frac{D_{n}^{2/\alpha}}{n^{2}}}< \infty$.
\end{thm}

\begin{proof}
 Let us fix $\varphi \in \mathcal{D}^{m-1, m-1}_{\mathbb{R}}(\mathbb{C}^{m}),$ and take a random sequence $\textbf{f}=\{f_{n}\}_{n=1}^{\infty}$ in $\mathcal{P}_{\infty}$. We apply the argument given in the proof of \cite[Corollary 1.3]{Shif}. Consider the non-negative current-valued random variables \begin{equation}\label{newr}
  X_{n}(\textbf{f}):= ([\widehat{Z_{f_{n}}}]- \mathbb{E}[\widehat{Z_{f_{n}}}], \varphi)^{2} \geq 0.\end{equation}Equivalent characterization of variance gives\begin{equation}\label{sumson}
     \int_{\mathcal{P}_{\infty}}{X_{n}(\textbf{f})d\mu_{\infty}(\textbf{f})}=\mathrm{Var}([\widehat{Z_{f_{n}}}], \varphi).
  \end{equation}
 Due to the convergence hypothesis and (\ref{sumson}), one has \begin{equation}\label{fin}
                                                     \sum_{n=1}^{\infty}{\int_{\mathcal{P}_{\infty}}{X_{n}(\textbf{f})d\mu_{\infty}(\textbf{f})}}=\sum_{n=1}^{\infty}{\mathrm{Var}([\widehat{Z_{f_{n}}}], \varphi)} < \infty.
                                                   \end{equation}By the relation (\ref{fin}) and the Beppo-Levi's theorem from the standard measure theory, we immediately get \begin{equation}\label{vars2}
                                        \int_{\mathcal{P}_{\infty}}{\sum_{n=1}^{\infty}{X_{n}(\textbf{f})}d\mu_{\infty}(\textbf{f})}=\sum_{n=1}^{\infty}{\mathrm{Var}([\widehat{Z_{f_{n}}}], \varphi)}< \infty.
                                      \end{equation}This implies that, for  $\mu_{\infty}$-almost surely, \, $\sum_{n=1}^{\infty}{X_{n}(\textbf{f})}$ is convergent, and therefore, \,$X_{n} \rightarrow 0$\,\, $\mu_{\infty}$-almost surely, which leads to the conclusion, by the definition (\ref{newr}) of random variables $X_{n}$,   $$\langle [\widehat{Z_{f_{n}}}], \varphi \rangle- \mathbb{E}\langle [\widehat{Z_{f_{n}}}], \varphi \rangle \rightarrow 0$$ $\mu_{\infty}$-almost surely. In light of this last information and Lemma \ref{expw}, we deduce that for $\mu_{\infty}$-almost every sequence $\{f_{n}\}$ in $\mathcal{P}_{\infty}$, $$[\widehat{Z_{f_{n}}}]\rightarrow dd^{c}V_{K, q}$$ in the weak* topology of currents.
\end{proof}

\begin{cor}\label{speco}
Let $K\subset \mathbb{C}^{m}$ be regular compact set, $f_{n} \in \mathcal{P}_{n}$ be as in (\ref{repch}) with $||p_{nj}||_{K} \leq 1$. Suppose that there is a countable dense set $\{z_{s}\}$ in $\mathbb{C}^{m}$ such that  \begin{equation}\label{count}
                                                                                \lim_{n\rightarrow \infty}{\frac{1}{2n}\log{\Gamma_{n}(z_{s})}}=V_{K}(z_{s}) \,\,\,\,s=1, 2, \ldots.
                                                                              \end{equation}  Then for $\mu_{\infty}$-almost every sequence $\textbf{f}=\{f_{n}\}_{n=1}^{\infty}$, \begin{equation}\label{equidc}
                                                      [\widehat{Z_{f_{n}}}] \longrightarrow dd^{c} V_{K}
                                                    \end{equation}in the weak* topology of currents as $n\rightarrow \infty$ provided that $\sum_{n=1}^{\infty}{\frac{D_{n}^{2/\alpha}}{n^{2}}}< \infty$.

\end{cor}

Corollary \ref{speco} is also a generalization of [\cite{Gnyz1}, Theorem 2.4]. From the proofs of the results in this section, we see that one can drop many probabilistic conditions including i.i.d.(independent and identically distributed) random coefficients meeting expectation and concentration requirements made in [\cite{BBL}, Theorem 4.1].

\section{Some Special Distributions}\label{sec6}

The moment condition (\ref{moment}) is verified by a number of other regularly encountered probability measures in the literature, such as the Gaussian, the Fubini-Study, locally moderate probability measures, probability measures with heavy tail and small ball probability. In this final section, we will focus our attention on these specific measures, especially on the Gaussian and the Fubini-Study probability measures. Additionally, we will examine random polynomials with independent and identically distributed coefficients, characterized by a probability distribution that possesses a bounded density and tails that decay logarithmically. We study these three types because their summability conditions are automatically satisfied as a consequence of the variance estimates.

\subsection{Gaussian and Fubini-Study measures}For $a=(a_{1}, \ldots, a_{n})\in \mathbb{C}^{n},$ the Gaussian measure on $\mathbb{C}^{n}$ is defined as  \begin{equation} \label{gaussme}d\mu_{n}(a)=\frac{1}{\pi^{n}}e^{-||a||^{2}}d\lambda_{n}(a),\end{equation}and the Fubini-Study measure $\mathbb{C}\mathbb{P}^{n}\supset \mathbb{C}^{n}$  is defined as follows: \begin{equation}\label{fubstudy} d\mu_{n}(a)=\frac{n!}{\pi^{n}}\frac{1}{(1+ ||a||^{2})^{n+1}}d\lambda_{n}(a).\end{equation}Here, $\lambda_{n}$ represents the Lebesgue measure on $\mathbb{C}^{n}$(identified with $\mathbb{R}^{2n}$).

In regards to these two measures, we cite two findings [\cite{BCM}, Lemma 4.8 and Lemma 4.10]: Given that  $\mu_{n}$ is the Gaussian, for every integer $n\geq 1$ and every $\alpha \geq 1$, we have \begin{equation}\label{gauss1}\int_{\mathbb{C}^{n}}{|\log{|\langle a, v \rangle|}|}^{\alpha}d\sigma_{n}(a)= 2 \int_{0}^{\infty}{r|\log{r}|^{\alpha} e^{-r^{2}} dr} \,\,\,\text{for}\,\,\text{all}\,\, v\in \mathbb{C}^{n},\,\,\,||v||=1;\end{equation} if $\mu_{n}$ is the Fubini-Study, then for every integer $n \geq 1$ and every $\alpha \geq 1$ \begin{equation}\label{gauss2}\int_{\mathbb{C}^{n}}{|\log{|\langle a, v \rangle|}|}^{\alpha}d\sigma_{n}(a)= 2\int_{0}^{\infty}{\frac{r|\log{r}|^{\alpha}}{(1+r^{2})^{2}}dr} \,\,\,\text{for}\,\,\text{all}\,\, v\in \mathbb{C}^{n},\,\,\,||v||=1.\end{equation}As is seen, the constant $D_n$ in the moment condition (\ref{moment}) turns into a universal constant independent of $n$, which we denote by $D_{0}$, for the Gaussian and the Fubini-Study probability measures. The fundamental property to calculate the above integrals is the unitary invariance of these two measures.

Let $K \subset \mathbb{C}^{m}$ be a locally regular compact set and $q$ a continuous real-valued weight function on it. Let us fix a general basis $\{p_{nj}\}_{j=1}^{d_{n}}$ with $\|p_{nj} e^{-nq}\|_{K} \leq 1$. Consider the random polynomials \begin{equation}\label{gafurep}f_{n}(z)=\sum_{j=1}^{d_{n}}{a_{j}^{(n)}p_{nj}(z)} \in \mathcal{P}_{n}\end{equation} where each $a_{j}^{(n)}$ are random variables whose joint probability distribution function is as in (\ref{gaussme}) and (\ref{fubstudy}). $(\mathcal{P}_{n}, \mu_{n}) $ and $(\mathcal{P}_{\infty}, \mu_{\infty})$ will be our probability spaces such that $\mu_{n}$ is the probability measure on $\mathbb{C}^{d_{n}}$ defined as in (\ref{gaussme}) and (\ref{fubstudy}). We prove some assertions simultaneously for both the Gaussian and the Fubini-Study probability measures. We will not give the unweighted versions because they simply can be obtainable by taking $K\subset \mathbb{C}^{m}$ to be regular and $q=0$.

\begin{lem}\label{expect}
Under the above conditions, we have $\mathbb{E}[\widehat{Z_{f_{n}}}]= \frac{1}{n}dd^c \log{\Gamma_{n}}.$
\end{lem}
\begin{proof}
Let us start with the relation (\ref{expd2}). \begin{equation}\label{expd3}
                                             \mathbb{E}\langle [\widehat{Z_{f_{n}}}], \varphi \rangle= \int_{\mathbb{C}^{d_{n}}}{\langle\frac{1}{2n}dd^{c} \log{\Gamma_{n}}, \varphi \rangle d\mu_{n}(a^{(n)})} + \frac{1}{n} \int_{\mathbb{C}^{m}}{\int_{\mathbb{C}^{d_{n}}}{\log{|\langle a^{(n)}, \beta^{(n)}(z)\rangle|}}d\mu_{n}(a^{(n)}) dd^{c}\varphi(z)}.
                                           \end{equation}Now we look at the second term in (\ref{expd3}). It can be written as follows using the notation of currents: \begin{equation*}
                                                       \langle \frac{1}{n} dd^{c}\big\{ \int_{\mathbb{C}^{d_{n}}}{\log{|\langle a^{(n)}, \beta^{(n)}(z) \rangle|}d\mu_{n}(a^{(n)})} \big\}, \varphi \rangle.
                                                     \end{equation*} By (\ref{gaussme}) and (\ref{fubstudy}) (with $\alpha=1$), the last expression is zero because the integral becomes a constant $D_{0}= 2 \int_{0}^{\infty}{r|\log{r}| e^{-r^{2}} dr}$ when $\mu_{n}$ is the Gaussian probability measure and $D_{0}=2\int_{0}^{\infty}{\frac{r|\log{r}|}{(1+r^{2})^{2}}dr}$ if $\mu_{n}$ is the Fubini-Study probability measure.
\end{proof}
As a result of Lemma \ref{expect}, we get the following corollary.
\begin{cor}\label{expg}
Let $K\subset \mathbb{C}^{m}$ and $q$ be given as before. Let  $f_{n} \in (\mathcal{P}_{n}, \mu_{n})$ be as in (\ref{gafurep}), where $(\mathcal{P}_{n}, \mu_{n})$ is either the Gaussian or the Fubini-Study probability spaces. Under the assumption that there is a countable dense set $\{z_{s}\}$ in $\mathbb{C}^{m}$ such that  \begin{equation}\label{count}
                                                                                \lim_{n\rightarrow \infty}{\frac{1}{2n}\log{\Gamma_{n}(z_{s})}}=V_{K, q}(z_{s}) \,\,\,\,s=1, 2, \ldots.,
                                                                              \end{equation} we have \begin{equation*}
                                                                                                     \mathbb{E}[\widehat{Z_{f_{n}}}] \longrightarrow dd^c V_{K, q}
                                                                                                   \end{equation*} in the weak* topology of currents when the degree $n\rightarrow \infty$.
\end{cor}

By Lemma \ref{expect}, the first term in (\ref{cordec}) will be cancelled by the second term of the variance. As for the second and the third terms in the expansion (\ref{cordec}), following exactly the same argument in the proof of Lemma \ref{expect}, one can easily see that they both are zero and so we are down to the fourth term to estimate, that is the term (\ref{cordec1}). Let us denote the triple integral related to (\ref{cordec1}) by $B_{0}$. By using the Cauchy-Schwarz inequality in (\ref{varterm3}), we can bound $B_{0}$ from above as follows with $\alpha=\theta=2$ \begin{equation}\label{csvar}
   B_{0} \leq \int_{\mathbb{C}^{m}}\int_{\mathbb{C}^{m}}{dd^{c}\varphi(z) dd^{c}\varphi(w)}\frac{1}{n^{2}} \Big\{ \int_{\mathbb{C}^{d_{n}}}{|\log{|\langle a^{(n)}, \beta^{(n)}(z) \rangle|}|^{2}d\mu_{n}(a^{(n)})}\Big\} \leq \frac{1}{n^{2}}\,C^{2}_{\varphi}\,D_{0},
\end{equation}where $D_{0}$  is as in (\ref{gauss1}) and (\ref{gauss2}) with $\alpha=2$. We have thus proved the following.

\begin{thm}\label{vargafub}
Let $K\subset \mathbb{C}^{m}$ and $q$ be as before. Let $(\mathcal{P}_{n}, \mu_{n})$ be either the Gaussian or the Fubini-Study probability space. For a test form $\varphi \in \mathcal{D}^{m-1, m-1}_{\mathbb{R}}(\mathbb{C}^{m})$, if $f_{n}$ in (\ref{gafurep}) is given, we have

\begin{equation}\label{varison}
   \mathrm{Var}{\langle [\widehat{Z_{f_{n}}}], \varphi \rangle} =O(\frac{1}{n^2}).
\end{equation}
\end{thm}

By the same method in Theorem \ref{equidc1} and combining Corollary \ref{expg} and Theorem \ref{vargafub}, we have the equidistribution theorem in these particular probability spaces. Notice that we do not need a summability condition anymore because of (\ref{varison}).

\begin{thm}\label{gafubeq} Under the same conditions of Corollary \ref{expg}, for $\mu_{\infty}$-almost every sequence $\textbf{f}=\{f_{n}\}_{n=1}^{\infty}$, the following holds
  \begin{equation*}
      [Z_{f_{n}}] \longrightarrow dd^c V_{K, q}
      \end{equation*}in the weak* topology of currents when $n\rightarrow \infty$.
\end{thm}

\begin{rem}
For other measures, such as locally moderate probability measures, probability measures with heavy tail and small ball probability, area measures of sphere the constants $D_{n}$ in (\ref{moment}) result in upper bounds relying on the degree $n$. For exposition of all these cases, one can look at the papers \cite{BG} and \cite{BCM}. We examine one more special case where the constants $D_{n}$ depend on $n$: Random polynomials with i.i.d. coefficients with a bounded probability distribution having a logarithmically decreasing tail estimate. For our purposes in this paper, we alloted the next subsection to this circumstance.

\end{rem}

\subsection{Random Holomorphic Polynomials with i.i.d. coefficients} \label{64}
Consider the random polynomials \begin{equation}\label{gafurep}f_{n}(z)=\sum_{j=1}^{d_{n}}{a_{j}^{(n)}p_{nj}(z)} \in \mathcal{P}_{n}\end{equation}
where $\{p_{nj}\}^{d_{n}}_{j=1}$ is a fixed general basis,  $\mu_{n}$ on the polynomial space $\mathcal{P}_{n}$ is, this time, induced by the probability distribution law $\mathbf{P}$ of the i.i.d. random coefficients $a^{n}_{j}$ in the representation (\ref{repch}) with a density $\varphi: \mathbb{C} \rightarrow [0, N]$ satisfying the property that there are constants $\delta>0$ and $\gamma >2m$ such that \begin{equation}\label{logtail}\mathbf{P}(\{z\in \mathbb{C}: \log|z|>R\})\leq \frac{\delta}{R^{\gamma}},\,\,\,\forall R\geq 1.\end{equation} This kind of density was studied in \cite{Bay16} and \cite{BCM}. This choice of probability distribution includes real or complex Gaussian distributions. The authors in \cite{BCM} (Lemma 4.15 there) show that the  measures $\mu_{n}$ verify the moment condition (\ref{moment}) with the upper bound \begin{equation}\label{dnmom} B d_{n}^{\alpha/ \gamma}\end{equation} ($B=B(N, \alpha, \gamma, \delta)$)\,for any constant $\alpha$ with $1\leq \alpha < \gamma$, which gives us that, under this setting, with the ideas we use in the previous section, the analogoues of Lemma \ref{expw}, Theorem \ref{main} and Theorem \ref{equidc1} can be seen to be true for asymptotically Bernstein-Markov probability measures. A similar probability distribution function was also considered by Bloom and Levenberg, see \cite{BL15} for further details.

  Since $d_{n}=\binom{n+m}{m}$, one can find a constant $C>0$ such that $d_{n} \leq C\,n^{m}$, so, by (\ref{dnmom}) and the inequality  $((d_{n})^{\alpha/\gamma})^{2/\alpha} \leq (C\,n^{m})^{\alpha/\gamma})^{2/\alpha}= C^{2/ \theta} n^{2m/ \gamma}$, we get $(B d^{\alpha/\gamma}_{n})^{2/\alpha} \leq B^{2 / \alpha} C^{2/ \gamma} n^{2m/\gamma}$. Now, by writing $D_{n}:=B (C n^{m})^{\alpha/\gamma}$, we see that \begin{equation}\label{sumi}\sum_{n=1}^{\infty}{\frac{D_{n}^{2/\alpha}}{n^{2}}} < \infty\end{equation} because $\gamma >2m$. We deduce an equidistribution result for codimension $1$.

  \begin{thm}
Let $K\subset \mathbb{C}^{m}$ be a locally regular compact set and $q$ a weight function on $K$. According to the above framework, suppose that there is a countable dense set $\{z_{s}\}$ in $\mathbb{C}^{m}$ such that \begin{equation}\label{count}
                                                                                \lim_{n\rightarrow \infty}{\frac{1}{2n}\log{\Gamma_{n}(z_{s})}}=V_{K, q}(z_{s}) \,\,\,\,s=1, 2, \ldots.,
                                                                              \end{equation}  we have, for $\mu_{\infty}$-almost every sequence $\textbf{f}=\{f_{n}\}_{n=1}^{\infty}$, the following holds \begin{equation}\label{equidc}
                                                      [\widehat{Z_{f_{n}}}] \longrightarrow dd^{c} V_{K}
                                                    \end{equation}in the weak* topology of currents as $n\rightarrow \infty$, where $k=1, 2, \ldots, m$.
\end{thm}

\section{Global Equidistribution}\label{s7}

In this final section we will be in the orthogonal setting to prove a more general version of [\cite{Bay16}, Theorem 1.1]. Instead of making a distinction between codimension one and codimensions greater than one, we examine the case of codimension $k$ in a unified manner, with $k$ spanning the range from $1$ to $m$.

\subsection{Deterministic Setting}
\subsubsection{K\"{a}hler geometric preliminaries}

Let $(X, \omega)$ be a connected compact K\"{a}hler manifold with $\dim_{\mathbb{C}}{X}=m$.
On this manifold, a holomorphic line bundle $L$ is defined by compiling complex lines $\{L_{x}\}_{x\in X}$ and constructing a complex manifold of dimension $1 + \dim_{\mathbb{C}}{X}$ with a projection map
$\pi : L \rightarrow X $ such that $\pi$ that assigns each line (or fiber) $L_{x}$ to $x$ is holomorphic. By using an open cover $\{U_{\alpha}\}$ of $X$, we can always locally trivialize $L$ through biholomorphisms $\Psi_{\alpha}: \pi^{-1}(U_{\alpha}) \rightarrow U_{\alpha} \times \mathbb{C}$ which map $L_{x} = \pi^{-1}(x)$ isomorphically onto $\{x\} \times \mathbb{C}$. The line bundle $L$ is then uniquely (i.e., up to
isomorphism) determined by these transition functions $g_{\alpha \beta}$, which are non-vanishing holomorphic functions on $U_{\alpha \beta}:=U_{\alpha} \cap U_{\beta}$ defined by $g_{\alpha \beta}= \Psi_{\alpha} \circ \Psi^{-1}_{\beta}|_{\{x\} \times \mathbb{C}}$. These functions $g_{\alpha \beta}$ satisfy the cocycle condition $g_{\alpha \beta} g_{\beta \gamma} g_{\gamma \alpha}=1$.

The cocycle condition for the transition functions yields that they define a cohomology class, denoted as
$[{g_{\alpha \beta}}] \in H^{1}(X, \mathcal{O}^{*})$. Here, $H^{1}(X, \mathcal{O}^{*})$ is the first sheaf cohomology group of the manifold $X$ with coefficients in the sheaf of non-zero holomorphic functions, denoted by  $\mathcal{O}^{*}$.  The exponential short exact sequence $0 \rightarrow \mathbb{Z} \rightarrow \mathcal{O} \rightarrow \mathcal{O}^{*} \rightarrow 0$ produces a mapping  $c_{1}: H^{1}(X, O^*) \rightarrow H^{2}(X, \mathbb{Z})$. $c_{1}(L, h)$ is defined by the image of $[{g_{\alpha \beta}}]$ under this mapping.

Let us denote the set of all plurisubharmonic functions on $U_{\alpha}$ by $\text{PSH}(U_{\alpha})$. Let $(L, h)$ be a holomorphic line bundle over $X$, where $h$ is provided by a collection of functions $\{\varphi_{\alpha}\}$ such that for any holomorphic frame $e$ of $L$ over $U_{\alpha}$, $\varphi_{\alpha} \in \mathcal{C}^{\infty}(U_{\alpha}) \cap \text{PSH}(U_{\alpha})$. We call the metric $h$ positive and smooth and $(L, h)$ a \textit{positively curved} line bundle. Smooth metrics always exist but positive metrics do not.

Let $(L, h)$ be a positively curved line bundle, where $h=\{e^{-\varphi_{\alpha}}\}$. The curvature of $h$ on each $U_{\alpha}$ is given by $\Theta_{h}=dd^{c} \varphi_{\alpha}$. It is well-defined real closed form on $X$ due to the relation $dd^{c} \log{|g_{\alpha \beta}|}=0$ on $U_{\alpha}$. By de Rham's isomorphism theorem, this curvature form represents the image of the \textit{first Chern class of $L$} under the mapping $i: H^{2}(X, \mathbb{Z}) \rightarrow H^{2}(X, \mathbb{R})$ provided by the inclusion $i:\mathbb{Z} \rightarrow \mathbb{R}$.

We assume $c_{1}(L, h)= \omega$. This condition is known as the \textit{prequantization} and the line bundle $(L, h)$ is called a \textit{prequantum line bundle}. Note that $X$ is also a projective manifold by Kodaira's embedding theorem (\cite{Huy}).
We will be interested in the global holomorphic sections of tensor powers of a prequantum line bundle $L$, defined by $L^{\otimes n}:= L\otimes \ldots \otimes L$ and for briefly we will write $L^{n}$. A global holomorphic section $s$ of $L^{n}$ is a family $\{s_{\alpha}\}$ of holomorphic functions such that the compatibility conditions $s_{\alpha}=g_{\alpha \beta}^{n} s_{\beta}$ on the overlapping open sets $U_{\alpha \beta}$ hold. The set of all global holomorphic sections is a finite dimensional vector space, denoted by $H^{0}(X, L^{n})$ and we write $\dim{H^{0}(X, L^{n})}=d_{n}$. The metric $h_{n}$ on the tensor power $L^{n}$ is induced by $h$ and $h_{n}:=h^{\otimes n}$. Let $s\in H^{0}(X, L^{n})$. The norm of $s$ will be denoted by $\|s\|_{h_{n}}$, this norm is defined on $U_{\alpha}$ by \begin{equation*}
                           \|s(x)\|_{h_{n}}:= |s_{\alpha}(x)|e^{-n\varphi_{\alpha}(x)}.
                         \end{equation*}Compatibility conditions guarantee that this definition does not depend on $\alpha$.

A function $\phi$ is said to be \textit{$\omega$-upper semicontinuous} if ($\omega$-u.s.c.) if $\phi + \tau$ is upper semicontinuous for every local potential $\tau$ of $\omega=dd^{c}\tau$. Any $\omega$-u.s.c. function $\phi \in L^{1}(X, \mathbb{R} \cup \{-\infty\})$ is called an \textit{$\omega$-plurisubharmonic} ($\omega$-psh) if $\omega + dd^{c}\phi \geq 0$ in the sense of currents. We denote the set of all $\omega$-psh functions by $\text{PSH}(X, \omega)$. $\omega$-plurisubharmonic functions are the most important and fundamental tools of global pluripotential theory on compact K\"{a}hler manifolds, for a nice and detailed investigation of $\omega$-plurisubharmonic functions, we invite the reader to consult the paper \cite{GZ}.

 Given $s\in H^{0}(X, L^{n})$, the current of integration $[Z_{s}]$ is defined exactly the same as in the polynomial case. The Poincar\'{e}-Lelong formula is, \begin{equation*}\label{pole1}
   [Z_{s}]= n c_{1}(L, h) + dd^{c}\log{||s||_{h_{n}}}= n \omega + dd^{c}\log{||s||_{h_{n}}},
 \end{equation*}which is exactly the same as (\ref{pole}) locally. If we consider the normalized current of integration $\frac{1}{n}[Z_{s}]$, we see that $c_{1}(L, h)=\omega$ and $[Z_{s}]$ are in the same cohomology class by $dd^{c}$-lemma.

\subsubsection{Pluripotential theory on compact K\"{a}hler manifolds}
 We call a subset $K$ of $X$\ PSH$(X, \omega$)-\textit{pluripolar}, if there exist   $v \in \text{PSH}(X, \omega)$ such that $K \subset \{x\in X: v(x)=-\infty \}$. We only take non-pluripolar compact subsets of $X$ since PSH$(X, \omega$)-\textit{pluripolar} sets characterize locally pluripolar sets in $X$, in other words, any locally pluripolar set is PSH$(X, \omega$)-\textit{pluripolar}. This is, as we mentioned in the beginning, the Josefson's theorem in this global setting of compact K\"{a}hler manifolds (e.g., \cite{GZ}, Theorem 7.2). By this equivalence, we will take into account only non-pluripolar compact sets. Let $q: K \rightarrow \mathbb{R}$ be a continuous function, which we call a \textit{weight } function. In \cite{Bay16}, inspired by \cite{GZ}, the weighted version of pluricomplex Green function is defined by
 \begin{equation}\label{wge}
  V_{K, q}=\sup\{\varphi \in \text{PSH}(X, \omega): \varphi(x) \leq q(x) \,\,\text{on}\,\,K \}.
 \end{equation}
 The weighted global extremal function is defined as the upper semicontinuous regularization $V^{*}_{K, q}$. As in the case of $\mathbb{C}^{n}$, if $K \subset X$ is a locally regular compact set in $X$, then $V_{K, q}$ is continuous, namely, $V_{K, q}=V^{*}_{K, q}$, we refer the reader to the Subsection 2.4 of \cite{Bay16}. The following weighted version of the Siciak-Zakharyuta theorem was proved in \cite{Bay16} by using the arguments of the proof of [\cite{GZ}, Theorem 6.2] and the proof of [\cite{BL07}, Lemma 3.2].

 \begin{thm}[\cite{Bay16}, Theorem 2.7]\label{wsz}
 Let $K \subset X$ be a locally regular compact set and $q:K \rightarrow \mathbb{R}$ be a continuous weight function. Write $\Phi_{n}(x):=\sup{\{\|s(x)\|_{h_{n}}: s\in H^{0}(X, L^{n}), \max_{y\in K}{\|s(y)\|_{h_{n}} \,e^{-nq(y)}} \leq 1\}}$. Then \begin{align}\label{szlbb}
                                                                                                                          V_{K, q}(x) &=  \sup{\{\frac{1}{n} \log{\|s(x)\|_{h_{n}}}: s\in \bigcup_{n=1}^{\infty}{H^{0}(X, L^{n})}, \,\, \max_{y\in K}\|s(y)\|_{h_{n}} \,\,e^{-nq(y)} \leq 1\}}\\
                                                                                                                           &= \lim_{n\rightarrow \infty}{\frac{1}{n}\log{\Phi_{n}(x)}} \label{seclim}
                                                                                                                        \end{align}uniformly on $X$.
 \end{thm}

For a basis $\{s_{nj}\}_{j=1}^{d_{n}}$ of $H^{0}(X, L^{n})$ with $$\max_{y\in K}{\|s_{nj}(y)\|_{h_{n}} \,e^{-nq(y)}} \leq 1,$$ as we have done for polynomials in Subsection \ref{s2}, we define the Bergman-type functions \begin{equation}\label{bts}
  \Gamma_{n}(x)=\sum_{j=1}^{d_{n}}{\|s_{nj}(x)\|^{2}_{h_{n}}}.
\end{equation}By Theorem \ref{wsz}, we get \begin{equation}\label{1982}\Gamma_{n}(x) \leq d_{n} (\Phi_{n}(x))^{2} \end{equation}and  $\|s_{nj}(x)\|_{h_{n}} \leq e^{n V_{K, q}(x)}$. It follows from the latter that \begin{equation*}\label{sver}
       \Gamma_{n}(x) \leq d_{n} e^{2n V_{K, q}(x)}.
     \end{equation*}By taking the logarithm of both sides and dividing by $n$ across the inequality, we have \begin{equation}\label{sver1}
     \frac{1}{n} \log{\Gamma_{n}(x)} \leq \frac{1}{n} \log{d_{n}} + 2V_{K, q}(x).
      \end{equation} Since $L$ is a positive line bundle on
the compact Kähler manifold $X$, by, for instance, [\cite{LO}, p.430], there exist constants
$B_{1}, B_{2} > 0$ such that \begin{equation} \label{big1}B_{1} n^{m} \leq d_{n} \leq B_{2} n^{m}. \end{equation}

 Let $K\subset X$ be a locally regular compact set and $q$ a continuous real-valued weight function on $K$. From this point onward, we follow \cite{BBL} to define the concept of a sequence of asymptotically Bernstein-Markov measures in the line bundle setting and prove an analogue of [\cite{BBL}, Proposition 2.8]. Let $\{\sigma_{n}\}_{n\geq 1}^{\infty}$ be a sequence of probability measures on $K$. Define the following inner product on $H^{0}(X, L^{n})$ \begin{equation}\label{ipa}
  \langle s_{1}, s_{2} \rangle_{L^{2}(e^{-2n\,q} \sigma_{n})}:= \int_{K}{\langle s_{1}, s_{2}\rangle_{h_{n}} e^{-2n \,q}d\sigma_{n}}.
\end{equation}for any $s_{1}, s_{2} \in H^{0}(X, L^{n})$, and (\ref{ipa}) induces the following norm on $H^{0}(X, L^{n})$: $$\|s\|_{L^{2}(e^{-2n\,q} \sigma_{n})}^{2}:= \int_{K}{\|s\|^{2}_{h_{n}} e^{-2n\,q} d\sigma_{n}}.$$ Let $R_{n}$ be the smallest positive constant so that
\begin{equation}\label{abm}
\max_{x\in K}{\|s(x)\|_{h_{n}} \,e^{-nq(x)}} \leq R_{n} \|s\|_{L^{2}(e^{-2n\,q} \sigma_{n})}=R_{n}\|s\,e^{-n\,q}\|_{L^{2}(\sigma_{n})}
\end{equation}for any $s\in H^{0}(X, L^{p})$. We take and fix an $L^{2}(e^{-2n \,q} \sigma_{n})$-orthogonal basis                                                                                                  $\{S_{nl}\}^{d_{n}}_{l=1}$ for  $H^{0}(X, L^{n})$ with  $\max_{x\in K}{\|S_{nl}(x)\|_{h_{n}} e^{-nq(x)}} = 1, \,l=1, \ldots, d_{n}$. Let $s \in H^{0}(X, L^{n})$ with $\max_{x\in K}{\|s(x)\|_{h_{n}} e^{-nq(x)}} \leq 1$. By considering the representation \begin{equation*}\label{repsec}
s(x)=\sum_{l=1}^{d_{n}}{c_{l}\frac{S_{nl}(x)}{\|S_{nl}\|_{L^{2}(e^{-2n \,q} \sigma_{n})}}},
 \end{equation*}we have \begin{equation*}\label{l2or}
                         \|s(x)\|^{2}_{h_{n}} \leq \Big(\sum_{l=1}^{d_{n}}{|c_{l}|^{2}} \Big) \Big( \sum_{l=1}^{d_{n}}{\frac{\|S_{nl}(x)\|_{h_{n}}^{2}}{\|S_{nl}\|^{2}_{L^{2}(e^{-2n \,q} \sigma_{n})}}} \Big)= \|s\|^{2}_{L^{2}(e^{-2n\,q} \sigma_{n})}\, \Big( \sum_{l=1}^{d_{n}}{\frac{\|S_{nl}(x)\|_{h_{n}}^{2}}{\|S_{nl}\|^{2}_{L^{2}(e^{-2n \,q} \sigma_{n})}}} \Big).
                       \end{equation*} Since $\|S_{nl}\|_{L^{2}(e^{-2n \,q} \sigma_{n})} \geq \frac{1}{R_{n}}$ by (\ref{abm}) and $\|s\|_{L^{2}(e^{-2n \,q} \sigma_{n})}^{2}  \leq \max_{x\in K}{\|s(x)\|_{h_{n}} e^{-nq(x)}} \leq 1$, it follows that
                       \begin{equation*}\label{latebb}
                          \|s(x)\|^{2}_{h_{n}} \leq R^{2}_{n} \Gamma_{n}(x).
                       \end{equation*} By taking the supremum for all $s \in H^{0}(X, L^{n})$ with the above properties and using $\Phi_{n}(x)=\sup{\{\|s(x)\|_{h_{n}}: s\in H^{0}(X, L^{n}), \max_{x\in K}{\|s_{nj}(x)\|_{h_{n}} \,e^{-nq(x)}} \leq 1\}}$, one gets \begin{equation}\label{szlde}
                        (\Phi_{n}(x))^{2} \leq R^{2}_{n} \Gamma_{n}(x).
                         \end{equation}

                       \begin{defn}
                       A sequence $\{\sigma_{n}\}_{n \geq 1}$ of probability measures on $K$ is said to be asymptotically weighted Bernstein-Markov for $K$ and $q$ if for all $n \geq 1$, $\max_{x\in K}{\|s(x)\|_{h_{n}} \,e^{-nq(x)}} \leq R_{n} \|s e^{-nq}\|_{L^{2}(\sigma_{n})}$ for all $s\in H^{0}(X, L^{n})$ with $\lim_{n\rightarrow \infty}{R^{\frac{1}{n}}_{n}}=1$.
                         \end{defn}
                        Then, by (\ref{seclim}), (\ref{1982}), (\ref{big1}) and (\ref{szlde}), we have proved the following   \begin{prop} \label{szlims}Let $K \subset X$ be a locally regular compact set, $q: K \rightarrow \mathbb{R}$ a continuous weight function. Let $\{\sigma_{n}\}_{n \geq 1}$ be a sequence of asymptotically weighted Bernstein-Markov measures for $K$ and $q$. Let $\{S_{nl}\}^{d_{n}}_{l=1}$ be an $L^{2}(\sigma_{n} e^{-2n\,q})$-orthogonal basis with $\max_{x\in K}{\|S_{nl}(x)\|_{h_{n}} e^{-nq(x)}}=1$ and $\Gamma_{n}(x)=\sum_{j=1}^{d_{n}}{\|S_{nj}(x)\|^{2}_{h_{n}}}$. Then
                          \begin{equation}\label{L2limsz}
                          \lim_{n\rightarrow \infty}{\frac{1}{2n} \log{\Gamma_{n}(x)}}=V_{K, q}(x)
                          \end{equation}uniformly on $X$.  \end{prop}

\subsection{Randomization}
 We shall be interested in the simultaneous zero locus \begin{equation}\label{zds}
 Z_{s^{1}_{n}, \ldots, s^{k}_{n}}:=\{x\in X: s^{1}_{n}(x)=\ldots=s^{k}_{n}(x)=0 \}
 \end{equation} and here \begin{equation}\label{indv}
  s^{j}_{n}=\sum_{l=1}^{d_{n}}{a^{j}_{nl}\,S_{nl}}\in H^{0}(X, L^{n}), \,\,\,j=1, 2, \ldots, k.\end{equation}

We randomize the spaces $H^{0}(X, L^{n})$, which is identified by $\mathbb{C}^{d_{n}}$ as in the previous sections, with the probability measures $\mu_{n}$ that satisfy the moment condition (\ref{moment}) and that do not charge pluripolar sets in $H^{0}(X, L^{n})$. We also consider $k^{th}$ product spaces, the infinite product of $k^{th}$ product probability spaces $(H^{0}(X, L^{n})^{k}, \mu^{k}_{n})$ and $$(\mathcal{H}^{k}_{\infty}, \mu_{\infty}^{k})=(\prod_{n=1}^{\infty}{H^{0}(X, L^{n})^{k}}, \prod_{n=1}^{\infty}{\mu^{k}_{n}}),$$ where the $k^{th}$ product measure is $\mu^{k}_{n}=\mu_{n}\times \ldots \times \mu_{n}$. Since $X$ is projective, by a consequence of \mbox{Kodaira's} embedding theorem, for $n$ sufficiently big, the base locus $Bs(H^{0}(X, L^{n}))$ will be empty. Also, by [\cite{BG}, Proposition 6.2 (a probabilistic Bertini's theorem) and Proposition 6.3], with \mbox{probability} one, we have that the zero locus of $\mathcal{G}^{k}_{n}=(s^{1}_{n}, \ldots, s^{k}_{n})$ is a compact complex submanifold of codimension $k$ for some large enough $n$ and that the random current of integration of the zero locus of the mapping $\mathcal{G}^{k}_{n}$, denoted by $[Z_{\mathcal{G}^{k}_{n}}]$, is well-defined via  $$[Z_{\mathcal{G}^{k}_{n}}]:= [Z_{s^{1}_{n}}] \wedge \ldots \wedge [Z_{s^{k}_{n}}].$$

Hence, almost all $k$-tuples of the random holomorphic sections are independent. By repeating the argument in [\cite{BG}, Theorem 3.3], we have \begin{equation}\label{inde}\mathbb{E}[Z_{\mathcal{G}^{k}_{n}}]= \mathbb{E}[Z_{s^{1}_{n}}] \wedge \ldots \wedge \mathbb{E}[Z_{s^{k}_{n}}], \end{equation}which is a positive closed $(k, k)$-current.

\begin{rem}
In the published version of \cite{BG}, the proof of Theorem 3.3 was
presented inductively, in keeping with the structure of some of the
arguments considered in the previous versions of \cite{BG}. However,
induction is not needed for this result. For greater
clarity, let us explain the direct argument here. We use a simple successive application of the Poincar\'e--Lelong formula
and the Fubini--Tonelli argument leading to \cite[(3.19)]{BG}. Applying
this argument successively to each of the independent sections yields
\[
\mathbb{E}_{\sigma_{p,1}\times\cdots\times\sigma_{p,k}}
[Z_{\Sigma_p^k}]
=
\bigwedge_{j=1}^{k}
\mathbb{E}_{\sigma_{p,j}}[Z_{s_p^j}],
\]
which proves the assertion. The four equalities appearing
in the proof of \cite[Theorem 3.3]{BG} should be understood merely as successive iterations of
the above argument, rather than as induction. The same argument of course works in the
present tensor product setting, so (\ref{inde}) follows.

\end{rem}

 Normalization for $k \geq 2$ will be given by $$[\widehat{Z_{\mathcal{G}^{k}_{n}}}]:=\frac{1}{n^{k}}[Z_{\mathcal{G}^{k}_{n}}].$$ By the same reasoning in [\cite{BG}, Theorem 4.2 and Theorem 1.1] and using Proposition \ref{szlims} where relevant, we get the following

\begin{lem}\label{sonex}
According to the above setup, under the conditions of Propositions \ref{szlims}, we have

\begin{equation}\label{expd78}
 \mathbb{E}[\widehat{Z_{\mathcal{G}^{k}_{n}}}] \longrightarrow (dd^c V_{K, q}+\omega)^{k}
\end{equation} in the weak* topology of currents as $n \rightarrow \infty$.
\end{lem}

\begin{thm}\label{varg2}

With the conditions of Proposition \ref{szlims} and the data given above, for any chosen form $\varphi \in \mathcal{D}^{m-k, m-k}_{\mathbb{R}}(X)$, there exists a constant $C_{\varphi}>0$ depending only on $\varphi$ such that $$\mathrm{Var}{\langle [\widehat{Z_{\mathcal{G}^{k}_{n}}}], \varphi \rangle} \leq C_{\varphi}^{2}\,D_{n}^{\frac{2}{\alpha}} \frac{1}{n^{2}}.$$

\end{thm}

The proof of Lemma \ref{sonex} is done exactly as in \cite[Theorem 4.2]{BG} whose proof is based on \cite[Proposition 3.5]{CLMM}, with certain modifications. With the same notation in \cite[Theorem 4.2]{BG}, we use first the factorization of difference of two positive closed $(k, k)$-currents $$\frac{\Theta^{k}}{n^{k}}:=\mathbb{E}[\widehat{Z_{\mathcal{G}^{k}_{n}}}]= \big[\omega + \frac{1}{2n}dd^{c}\log{\Gamma_{n}} + \frac{1}{n}dd^{c} (\int_{f_{n} \in \mathcal{P}_{n}}{{\log{|\langle a^{(n)}, \beta^{(n)}(z)\rangle|}}d\mu_{n}(f_{n})}) \big]^{k}$$ and $(\omega + dd^{c}V_{K, q})^{k}$ as follows: For $\varphi \in \mathcal{D}_{\mathbb{R}}^{m-k, m-k}(X)$,
\begin{equation}\label{sonii}
 \langle \frac{\Theta^{k}}{n^{k}} - (\omega + dd^{c}V_{K, q})^{k}, \varphi \rangle = (\frac{1}{2n}dd^{c}\log{\Gamma_{n}} - dd^{c}V_{K, q})) \wedge \upsilon_{n}, \varphi \rangle + \langle \frac{1}{n}dd^{c} (\int_{f_{n} \in \mathcal{P}_{n}}{{\log{|\langle a^{(n)}, \beta^{(n)}(z)\rangle|}}d\mu_{n}(f_{n})})\wedge \upsilon_{n} , \varphi \rangle,
 \end{equation} where $$\upsilon_{n}:= \sum_{j=0}^{k-1} \frac{\Theta_{n}^{j}}{n^{j}} \wedge (\omega + dd^{c}V_{K, q})^{k-1-j}.$$ Then we estimate (\ref{sonii}) by using the uniform convergence of $\frac{1}{2n}\log{\Gamma_{n}}$ to $V_{K, q}$ (Proposition \ref{szlims}), the logaritmic moment condition (\ref{moment}) and volume inequality \cite[(2.2)]{BG}. One key step in the rest of the proof is that $\Theta$ is a positive current unlike the proof of \cite[Proposition 3.5]{CLMM}, which deals with differential forms only, so one must use an anticommutativity convention \cite[(2.1)]{BG}.

 The proof of Theorem \ref{varg2} relies on \cite[Theorem 3.4]{BG}, which is a more general version of \cite[Theorem 1.1]{BG} and is proved also to conclude \cite[Theorem 1.1]{BG} and uses conditional variances for individual codimension $1$ divisior currents.
 
 A subtle point, which was not fully apparent to us at the time of writing the present paper, is that a restriction-based induction of Shiffman is not available in the present non-Gaussian framework without additional assumptions. Indeed, after passing to an intermediate zero set, one is naturally led to induced coefficient measures on the corresponding restricted spaces of sections, and the moment condition need not be preserved in a form compatible with the inductive argument. For this reason, the argument in \cite[Theorem 3.4]{BG} avoids restricted coefficient measures altogether and proceeds instead through successive codimension-one variance estimates.

\vspace{4mm}
By mimicking the proof of \cite[Theorem 4.3]{BG} with Lemma \ref{sonex} and Theorem \ref{varg2}, we derive the following equidistribution result in codimensions $k \geq2$ in the global framework. This theorem improves the results of \cite[Theorem 1.1]{Bay16} according to the probabilistic setting and asymptotically Bernstein-Markov measures used in the current paper. Theorem \ref{co26} also proves that \cite[Theorem 1.1]{Bay16} is true for non-homogeneous projective manifolds, affirmatively answering the question raised in the same paper.

\begin{thm} \label{co26}
Under the assumptions of Proposition \ref{szlims}, if $\sum_{n=1}^{\infty}{\frac{D_{n}^{2/\alpha}}{n^{2}}}< \infty$, then for $\mu^{k}_{\infty}$-almost every sequence $\{\mathcal{G}^{k}_{n}\}$, \begin{equation}\label{equidc}
                                                      [\widehat{Z_{\mathcal{G}^{k}_{n}}}] \longrightarrow (dd^{c}V_{K,q} + \omega)^{k}
                                                    \end{equation}in the weak* topology of currents as $n\rightarrow \infty$.
\end{thm}

If we take $\mu_{n}$ as in Subsection \ref{64} and use the necessary information there, we have the next corollary. This is a more general version of \cite[Theorem 1.1]{Bay16} based on the aforementioned aspects.

\begin{cor}
 Given the hypotheses of Proposition \ref{szlims}, we have, for $\mu^{k}_{\infty}$-almost every sequence $\{\mathcal{G}^{k}_{n}\}$, \begin{equation}\label{equidc}
                                                      [\widehat{Z_{\mathcal{G}^{k}_{n}}}] \longrightarrow (dd^{c}V_{K,q} + \omega)^{k}
                                                    \end{equation}in the weak* topology of currents as $n\rightarrow \infty$.
\end{cor}
All of the results above in this section are valid in the unweighted case $q=0$ when $K \subset X$ is regular. The statements and proofs are identical to the weighted ones, so we will not include them. As a last note, if $K=X$ and $q=0$, then, from the considerations of \cite[Subsection 9.4]{GZ1}, it follows that $V_{X, 0} \equiv 0$ and $dd^{c} V_{K, q} + \omega$ is equal to $\omega$, hence, Theorem \ref{co26} is a more general form of  \cite[Theorem 1.1]{SZ99}, according to the probabilistic setting considered in this paper.

\vspace{5mm}
\textbf{Acknowledgement:} The author is grateful to the anonymous referee for his/her careful review, suggestions and corrections, which contributed to improving the article’s presentation. The author also sincerely thanks Afrim Bojnik for discussions regarding the subject matter of this work.

\vspace{5mm}

%%%%%%%%%%%%%%%%%%%%%%%%%%%%%%%%%%%%%%%%%%%%%%%%%%%%%%%%%%%

{}

 \end{document}